\documentclass[11pt]{amsart}
\usepackage{amsmath,amssymb,amsthm,amscd,verbatim}
\usepackage{graphicx}
\usepackage{epsfig}
\usepackage{stmaryrd}
\numberwithin{equation}{section}

\newtheorem{thm}[equation]{Theorem}

\newtheorem{lem}[equation]{Lemma}

\newtheorem{question}[equation]{Question}

\newtheorem*{thm*}{Theorem}

\theoremstyle{remark}

 \newcommand{\ba}{\mathbf{a}}
 
  \newcommand{\bk}{\mathbf{k}}

    \newcommand{\Inj}{{\operatorname{Inj}}}

\newcommand{\Aut}{\operatorname{Aut}}

\def\Ex {{\mathbb E}}

\begin{document}

\title {The inducibility of blow-up graphs}

\author{Hamed Hatami}
\address{School of Computer Science, McGill University, Montreal, Canada.}
\email{hatami@cs.mcgill.ca}

\author{James Hirst}
\address{School of Computer Science, McGill University, Montreal, Canada.}
\email{james.hirst@mail.mcgill.ca}

\author{Serguei Norine}
\address{Department of mathematics, McGill University, Montreal, Canada.}
\email{snorin@math.mcgill.ca}

\thanks{Hamed Hatami was supported in part by NSERC and FQRNT. James Hirst was supported in part by NSERC}

\begin{abstract}
The blow-up of a graph is obtained by replacing every vertex with a finite collection of copies so that the copies of two vertices are adjacent if and only if the originals are. If every vertex is replaced with the same number of copies, then the resulting graph is called a balanced blow-up.

We show that any graph which contains the maximum number of induced copies of a sufficiently large balanced blow-up of $H$  is itself essentially a blow-up of $H$. This gives an asymptotic answer to a question in~\cite{MR1326923}.
\end{abstract}

\maketitle

\newline
{{\sc Keywords:} induced subgraphs; inducibility; blow-up

\section{Introduction \label{sec:intro}}

What is the maximum number of induced copies of a given graph $H$ in any graph on $n$ vertices? This basic question has been studied previously in~\cite{MR1326923,MR866942,MR867727,MR0401552,MR1292981,James}, and its answer is known for a handful of small graphs and certain complete multipartite graphs.
Obtaining a general answer seems to be difficult. Even the case of the path on four vertices is not resolved. In this paper we provide an asymptotic answer to this question
for the class of graphs described below.

The blow-ups of $H$ provide natural candidates for graphs that contain many induced copies of $H$.
For a positive integer vector $\bk \in \mathbb{Z}_+^{V(G)}$, the $\bk$-\emph{blow-up} of $H$, denoted by $H^{(\bk)}$, is the graph obtained by replacing every
vertex $v$ of $H$ with $\bk(v)$ different vertices where a copy of $u$ is adjacent to a copy of $v$ in the blow-up graph if and only if $u$ is adjacent to $v$ in $H$. When all $\bk(v)$ are equal to  some positive integer $k$, the corresponding blow-up is called {\em balanced} and denoted simply by $H^{(k)}$.

For two graphs $H$ and $G$, the \emph{induced density} of $H$ in $G$, denoted by  $i(H,G)$ is the number of induced copies of $H$ in $G$ divided by ${|V(G)| \choose |V(H)|}$. Let $i(H,n)$ denote the maximum induced density of $H$ in any graph on $n$ vertices. Modern extremal graph theory more often focuses on understanding the asymptotic behavior of such functions rather than determining their values on every given $n$. The asymptotic behavior of $i(H,n)$ is captured by the \emph{inducibility} of $H$  defined as  $i(H):=\lim_{n \to \infty} i(H,n)$. It is straightforward to see that this limit always exists.

The blow-ups of the graph $H$ provide natural candidates for graphs with largest induced density of $H$. Note that graphs with maximum edge density are always complete, but the blow-ups of a one-edge graph are complete bipartite graphs. As this simple example already illustrates, it is not in general true that the maximal graphs are always the blow-ups of $H$.

In contrast with the above example, our main result, Theorem~\ref{thm:main}, says that for every graph $H$, the inducibility of $H^{(h)}$ is essentially achieved by blow-ups of $H$. More precisely, for sufficiently large $h$, there exists vectors $\bk_n \in \mathbb{Z}_+^{V(H)}$ with $|V(H^{(\bk_n)})|=n$ such that $$i(H^{(h)}) = \lim_{n \to \infty} i(H,H^{(\bk_n)}).$$

Note that $K_r^{(h)}$, the $h$-blow-up of the complete graph on $r$ vertices, is the  complete $r$-partite graph where each part has exactly $h$ vertices.
Bollob\'as, Egawa, Harris and Jin~\cite{MR1326923} showed that for sufficiently large $h$, the maximal graphs for $K_r^{(h)}$ are blow-ups of $K_r$. They asked for which graphs $H$ does there exist a constant $h$ such that for sufficiently large $n$, the graph on $n$ vertices that contains the maximum number of induced copies of $H^{(h)}$ is a blow-up of $H$? Note that Theorem~\ref{thm:main} gives an asymptotic answer to this question.

\section{Preliminaries \label{sec:prel}}
All graphs in this paper are finite and loopless. We denote vectors with bold font, e.g. $\ba = (\ba(1), \ba(2), \ba(3))$ is a vector
with three coordinates. For every positive integer $k$, let $[k]$ denote the set
$\{1, \ldots, k\}$. The \emph{adjacency matrix} of  a graph $G$, denoted by $A_G$, is a zero-one matrix whose rows and columns are indexed by
vertices of $G$ and $A_G(u,v)=1$ if and only if $u$ is adjacent to $v$. For a graph $G$ and a subset $S \subseteq V(G)$,
the subgraph of $G$ induced by $S$ is denoted by $G[S]$.

A \emph{homomorphism} from a graph $H$ to a graph $G$ is a mapping $\phi:V(H) \rightarrow V(G)$ such that if
$uv$ is an edge of $H$, then  $\phi(u)\phi(v)$ is an edge of $G$. A \emph{strong homomorphism} from a graph $H$ to a graph $G$ is a map  $\phi:V(H) \to V(G)$
such that $uv \in E(H)$ if and only if $\phi(u)\phi(v) \in E(G)$.  An injective strong homomorphism is called an \emph{embedding}.

A bijective strong homomorphism from $H$ to $G$ is  an \emph{isomorphism} between $H$ and $G$.
An isomorphism from $G$ to itself is called an \emph{automorphism}. Note that the set of automorphisms of a graph $G$ with the composition operator constitutes a group $\Aut(G)$ which is called the \emph{automorphism group} of $G$.

Two vertices of a graph are called \emph{twins} if they have the same set of neighbors. A graph $G$ is called \emph{twin-free} if no two distinct vertices of $G$ are twins. Every strong homomorphism from a twin-free graph into any other graph is necessarily an embedding.  Given a graph $G$, define the equivalence relation $\sim$ on $V(G)$ by letting $u \sim v$ if and only if $u$ and $v$ are twins. Clearly, all classes of this relation are independent sets in $G$ and we can form the factor-graph $\tilde{G} := G/ \sim$. Note that $G$ is twin-free and $G = \tilde{G}^{(\bk)}$, where $\bk$ is the vector of cardinalities of equivalence classes. This factorization of every finite graph as a blow-up of a twin-free graph is unique. Note that $\tilde{G}$ is the unique twin-free graph such that there exists a surjective strong homomorphism from $G$ to $\tilde{G}$.

For two graphs $H$ and $G$, define the \emph{strong homomorphism density} of $H$ in $G$, denoted by $s(H;G)$, to be the probability that a random map from $V(H)$ to $V(G)$ is a strong homomorphism. Given a graph $H$ and a sufficiently large $h$, we are interested in determining the supremum of  $s(H^{(h)};G)$ over all finite graphs $G$. It turns out that this supremum is not necessarily attained by any graph $G$. The common approaches to remedy this is either to work with a generalization of graphs called \emph{graphons} as in~\cite{MR2274085}, or to directly work with homomorphisms from the so called \emph{flag algebras} as in~\cite{MR2371204}. However for our purposes, it suffices to work with the simpler objects, namely weighted graphs.

\subsection{Weighted graphs}\label{sec:weighted}

For a finite set $S$, let $\mathcal{M}(S)$ denote the set of all probability measures  on $S$ with the property that every element of $S$ has non-zero measure.
For $\mu \in \mathcal{M}(S)$, let $\alpha(\mu):=\min_{v \in S} \mu(v)$. A \emph{weighted graph} $G^{\mu}$ is a graph $G$ together with a probability measure $\mu \in \mathcal{M}(V(G))$.

We define $s(H;G^{\mu})$ as the probability that a random map $\psi: V(H) \to V(G)$ is a strong homomorphism, where $\psi$ maps the vertices of $H$ into $V(G)$ independently according to the distribution $\mu$. Note that if $\mu$ is the uniform distribution, then this new definition coincides with our earlier definition of a strong homomorphism density.

We say that two weighted graphs $G_1^{\mu_1}$ and $G_2^{\mu_2}$ are equivalent, denoted  by $G_1^{\mu_1} \sim G_2^{\mu_2}$, if there exists an isomorphism $\theta:\tilde{G_1} \to \tilde{G_2}$ such that for every vertex $v \in \tilde{G_1}$, we have
$$\mu_1\left(\{u: \mbox{$u$ is a twin of $v$ in $G_1$}\}\right) =\mu_2\left(\{u: \mbox{$u$ is a twin of $\theta(v)$ in $G_2$}\}\right).$$
Note that $\sim$ is an equivalence relation, and furthermore if $G_1^{\mu_1} \sim G_2^{\mu_2}$, then for every graph $H$ we have $s(H;G_1^{\mu_1})=s(H;G_2^{\mu_2})$. It follows from standard results about graphons~\cite[Theorem 9]{ElekSzegedy} that if for every graph $H$ we have $s(H;G_1^{\mu_1})=s(H;G_2^{\mu_2})$, then $G_1^{\mu_1} \sim G_2^{\mu_2}$.

Two weighted graphs $F^{\mu_1}$ and $G^{\mu_2}$ are called \emph{commensurable} if  $V(F)=V(G)$ and $\mu_1=\mu_2$. The $d_1$ distance
of two weighted graphs  $G_1^{\mu_1}$ and $G_2^{\mu_2}$ is defined as
$$d_1(G_1^{\mu_1},G_2^{\mu_2}) := \inf \sum_{u,v \in V(F_1)}\mu(u)\mu(v) |A_{F_1}(u,v)-A_{F_2}(u,v)|,$$
where the infimum is over commensurable weighted graphs $F_1^{\mu} \sim G_1^{\mu_1}$ and $F_2^{\mu} \sim G_2^{\mu_2}$. It is easy to see that the infimum is always attained and thus can be replaced with a minimum. A simple union bound shows that for every $H$,
\begin{equation}
\label{eq:continuity}
\left| s(H;G_1^{\mu_1}) - s(H;G_1^{\mu_1}) \right| \le |V(H)|^2 d_1(G_1^{\mu_1},G_2^{\mu_2}).
\end{equation}

Let $G^{\mu}$ and $F^{\mu}$ be commensurable weighted graphs. We say that $v_0 \in V(G)$ is \emph{$\epsilon$-regular} with respect to $(G,F,\mu)$, if for some $v \in V(F)$, $$\mu\left(\{ u \in V(G) : A_G(v_0,u) \neq A_F(v,u)\}\right) \leq \epsilon.$$
Roughly speaking, this means that if $v_0$ is not $\epsilon$-regular, then the neighborhood of $v_0$ differs greatly from the neighborhood of every  vertex in $F$.

\subsection{Partially labeled graphs}
Let us recall some notations related to quantum graphs defined by Freedman, Lov\'asz, and Schrijver~\cite{MR2257396}. Analogous but slightly different definitions are given by Razborov in the context of flag algebras~\cite{MR2371204}.  For a nonnegative integer $k$, a $k$-\emph{partially labeled graph} is a graph in which $k$ of the vertices are labeled by distinct natural numbers $1,\ldots,k$ (there may be any number of unlabeled vertices).  Let ${\mathcal F}_k$ denote the set of all $k$-partially labeled graphs. Note that  ${\mathcal F}_0$ is the set of all finite graphs with no labels.

Let $G^\mu$ be a weighted graph.
We  extend the definition of strong homomorphism density to partially labeled graphs. Let $H$ be a $k$-partially labeled graph, and $\phi:[k] \rightarrow V(G)$ be a map. Identifying the labeled vertices in $H$ with their labels, we can think of $\phi$ as a map from the labeled vertices of $H$ into $V(G)$.  Hence we can define $s(H,\phi;G^\mu)$ as the probability that a random extension of $\phi$ to a map from $V(H)$ to $V(G)$ is a strong homomorphism, where the unlabeled vertices are mapped randomly and independently according to the measure $\mu$. Note that
\begin{equation}
\label{eq:averaging}
s(H;G^\mu) = \Ex_\phi \left[ s(H,\phi;G^\mu) \right],
\end{equation}
where in the expectation, $\phi$ is a random map from $[k]$ to $V(G)$ chosen according to the measure $\mu$.

A \emph{$k$-quantum graph} is a formal linear combination of $k$-partially labeled graphs. More formally, if $H_1,\ldots,H_m$ are $k$-partially labeled graphs, and $a_1,\ldots,a_m$ are real numbers,  then the formal sum $a_1 H_1 + \ldots + a_m H_m$ is called a $k$-quantum graph.

The function $s(\cdot,\phi; G^\mu)$ extends linearly to the set of all $k$-quantum graphs. More precisely for a $k$-quantum graph $f:=a_1H_1+\ldots+a_mH_m$, we have $$s(f,\phi; G^\mu):=a_1 s(H_1,\phi; G^\mu)+\ldots+a_m s(H_m,\phi; G^\mu).$$

For a graph $H$, define the $1$-quantum graph $\partial H$ as
$$\partial H := \sum_{u \in V(H)} H_{u},$$
where $H_{u}$ is the $1$-partially labeled graph obtained from $H$ by labeling
the vertex $u$ with label $1$.

The choice of the notation $\partial H$ is explained by the following observation. For fixed $H$ and $G$,  $s(H;G^\mu)$  is a polynomial in the variables $\{\mu(v)\}_{v \in V(G)}$,  and for very $v_0 \in V(G)$, we have
\begin{equation}
\label{eq:partial}
\frac{\partial}{\partial \mu(v_0)}s(H;G^\mu)=s(\partial H, 1 \mapsto v_0; G^\mu).
\end{equation}

\section{Main Results}

It is not difficult to see that
\begin{equation}
\label{eq:inducibilityHom}
i(F)=\frac{1}{|\Aut(F)|} \sup s(F;G^\nu),
\end{equation}
where the supremum is taken over all weighted graphs $G^\nu$. Note that this supremum can equivalently be taken over unweighted graphs.

Our first theorem shows that if $F$ is a sufficiently large balanced blow-up of a fixed graph $H$, then the supremum in the right-hand side of (\ref{eq:inducibilityHom}) is attained when $G=H$ for an appropriate choice of $\nu$. Our second theorem classifies the graphs $H$ for which this $\nu$ can be chosen to be uniform.

\begin{thm}
\label{thm:main}
Let $H$ be a graph. There exists a positive integer $h_0$ such that for every $h \ge h_0$,  there exists a distribution $\nu \in \mathcal{M}(V(H))$ with
$$s(H^{(h)};H^{\nu}) = \sup_{G} s(H^{(h)};G),$$
where the supremum is over all finite graphs $G$.
\end{thm}

Let $H$ be a graph and $\bk$ be such that $H = \tilde{H}^{(\bk)}$. We call $H$ \emph{balanced} if $\bk$ is invariant under every automorphism of $\tilde{H}$.

\begin{thm}
\label{thm:balanced}
A graph $H$ is balanced if and only if there exists a positive integer $h_0$ such that for every $h \ge h_0$,
$$s(H^{(h)};H)= \sup_G s(H^{(h)};G),$$
where the supremum is over all finite graphs $G$. Furthermore if $s(H^{(h)};H)=s(H^{(h)};G^\nu)$, then $H \sim G^\nu$.
\end{thm}

\section{Technical lemmas~\label{sec:lemmas}}
In this section we state and prove a few easy and technical observations which are used in the proofs of Theorems~\ref{thm:main}~and~\ref{thm:balanced}.
For a nonzero vector $\bk \in \mathbb{R}_+^k$, let the distribution $\mu_{\bk}$ be defined by $\mu_{\bk}(i):=\frac{\bk(i)}{\|\bk\|_1}$.  Given a finite set $S$, a distribution $\mu \in \mathcal{M}(S)$, and a vector $\ba \in \mathbb{R}_+^S$, we define
\begin{equation}
\label{eq:s_ba}
p_{\ba}(\mu):=\prod_{v \in S}\mu(v)^{\ba(v)}.
\end{equation}
The arithmetic-geometric inequality implies that for every $\ba \in \{x \in \mathbb{R}:x \ge 1\}^S$, we have
\begin{equation}
\label{eq:maximizing_product}
\sup_{\mu \in \mathcal{M}(S)}  p_{\ba}(\mu) = p_{\ba}(\mu_\ba).
\end{equation}

Let $J$ be a graph with edge homomorphism density at most $\delta$ and maximum degree at most $\epsilon |V(J)|$.
The next lemma obtains an easy upper-bound on the probability that $r$ randomly and independently chosen vertices of $J$  contain a copy of $K_{\ell,\ell}$ (or $K_{1,s}$) as a subgraph.
\begin{lem}\label{lem:technical1}
Let $\delta,\epsilon>0$ and $J^\mu$ be a weighted graph with
$s(K_2;J^\mu) \leq \delta$ and $\mu(\{u:uv \in E(J)\}) \leq \epsilon $
for every $v \in V(J)$.  Let a multiset $Z$ consisting of $r$ vertices
of $J$ be chosen according to $\mu$ and independently at random. Then
\begin{equation}\label{eq:bound1}
\Pr[ \mbox{$J[Z]$  contains a subgraph isomorphic to $
K_{\ell,\ell}$}] \leq 3^{r}\delta^{\ell}
\end{equation}
for every integer $\ell \geq 1$, and
\begin{equation}\label{eq:bound2}
\Pr [ \mbox{$J[Z]$ contains a vertex of degree} \geq s] \leq 3^r
\delta \epsilon^{s-1}
\end{equation}
for every integer $s \geq 1$.
\end{lem}
\begin{proof}
For the proof of (\ref{eq:bound1}) note that there are at most $3^r$
ways to select among elements of $Z$ the parts $(A,B)$ of
$K_{\ell,\ell}$: Each element of $Z$ can belong to $A$, to $B$, or to $Z - A
-B$. For every choice of $A=\{a_1,a_2, \ldots, a_\ell\}$ and $B=\{b_1,b_2,
\ldots, b_\ell\}$, we have $\Pr[a_ib_i \in E(J)] \leq \delta$ and the
corresponding events are independent for $1 \leq i \leq \ell$. The bound
 (\ref{eq:bound1})  follows.

The proof of (\ref{eq:bound2}) is similar. There at most $3^r$ ways of
specifying a vertex $a  \in Z$ and a set $b_1,b_2, \ldots, b_s$ of its
neighbors. We have $\Pr[ab_1 \in E(J)] \leq  \delta$, and $\Pr[ab_i
\in E(J) \: | \: a,b_1,\ldots, b_{i-1}] \leq \epsilon$ for $2 \leq i
\leq s$, implying (\ref{eq:bound2}).
\end{proof}
The next lemma plays a key role in the proofs of Theorems~\ref{thm:main}~and~\ref{thm:balanced}.
\begin{lem}\label{lem:technical2}
Let $\tilde H$ be a twin-free graph with $|V(\tilde H)|=k$, and let $\bk \in \mathbb{N}^{V(\tilde H)}$.  Let $G =\tilde{H}^{(n \bk)}$ and let $\psi: V(G) \to V(\tilde H)$ be a map. Then for every $\gamma >0$ either
\begin{itemize}
\item[{\bf (a)}] there exists a set $X \subseteq V(G)$ with $|X| \leq \gamma |V(G)|$ so that $\psi |_{V(G) - X}$ is a strong homomorphism, or
\item[{\bf (b)}] there exist sets $Y_1,Y_2 \subseteq V(H')$ with $|Y_1|,|Y_2| \geq \frac{\gamma \alpha(\mu_\bk)}{k}|V(G)|$ so that  for all $y_1 \in Y_1$, $y_2 \in Y_2$ we have $A_G(y_1,y_2) \neq A_{\tilde H}(\psi(y_1),\psi(y_2))$.
\end{itemize}
\end{lem}
\begin{proof}
We can assume that $\gamma \le 1$ as otherwise {\bf (a)} is trivial.
Since  $G =\tilde{H}^{(n \bk)}$, there exists a  strong homomorphism $\pi: V(G) \to V(\tilde H)$. For each $v \in V(\tilde H)$, one can find $\sigma(v) \in V(\tilde{H})$ so that
\begin{equation}
\label{eq:lem_technical2}
|\pi^{-1}(v) \cap  \psi^{-1}(\sigma(v))| \geq \frac{|\pi^{-1}(v)|}{k} = \frac{n \bk(v)}{k} \ge \frac{\alpha(\mu_\bk) |V(G)|}{k}.
\end{equation}
Now, if $\sigma : V(\tilde H) \to V(\tilde H)$ is not a strong homomorphism,  there exist $v_1,v_2 \in V(\tilde H)$ so that $A_{\tilde H}(v_1,v_2)\neq A_{\tilde H}(\sigma(v_1),\sigma(v_2))$. Then {\bf (b)} is satisfied by $Y_i:=|\pi^{-1}(v_i) \cap \psi^{-1}(\sigma(v_i))|$ for $i=1,2$ as $|Y_i| \geq \frac{\alpha(\mu_\bk)}{k}|V(G)|$.

Therefore we assume that $\sigma$ is a strong homomorphism and consequently an automorphism of $\tilde H$ since $\tilde H$ is twin-free. Without loss of generality, we assume $\sigma$ is the identity. If $|\pi^{-1}(v) \cap \psi^{-1}(v)| \geq (1 - \gamma) |\pi^{-1}(v)|$ for every $v \in V(\tilde{H})$, then $\pi$ differs from $\psi$ on at most $\gamma|V(G)|$ vertices and {\bf (a)} holds. Therefore, we assume that  $ |\pi^{-1}(v_0) \setminus \psi^{-1}(v_0)| \geq  \gamma |\pi^{-1}(v_0)|$ for some $v_0 \in V(\tilde H)$. Then there exists $v_1 \neq v_0$ satisfying $$|\pi^{-1}(v_0) \cap \psi^{-1}(v_1)|  \geq \frac{\gamma}{k} |\pi^{-1}(v_0)| \ge \frac{\gamma\alpha(\mu_\bk)}{k}|V(G)|.$$ Set $Y_1 := \pi^{-1}(v_0) \cap \psi^{-1}(v_1)$. Since $\tilde{H}$ is twin-free,  $v_1$  is not a twin of $v_0$, and thus there exists $v_2 \in V(\tilde H)$ so that  $A_{\tilde H}(v_0,v_2) \neq A_{\tilde H}(v_1,v_2)$. Set $Y_2 := \pi^{-1}(v_2) \cap \psi^{-1}(\sigma(v_2))=\pi^{-1}(v_2) \cap \psi^{-1}(v_2)$. Thus {\bf (b)} holds as (\ref{eq:lem_technical2}) shows that $|Y_2| \geq \frac{\alpha(\mu_\bk) |V(G)|}{k}$,  and furthermore for $y_1 \in Y_1$ and $y_2 \in Y_2$, we have $$A_G(y_1,y_2) = A_{\tilde H}(v_0,v_2) \neq A_{\tilde H}(v_1,v_2) = A_{\tilde H}(\psi(y_1),\psi(y_2)).$$
\end{proof}

\section{Proof of Theorem~\ref{thm:main}}

We want to show that for sufficiently large $h$, the density $s(H^{(h)};G^\mu)$ cannot be larger than $s(H^{(h)};H^\nu)$ for every $\nu$. The first step  is to show that if it is, then the global structure of $G^\mu$ is close to $H$.

\begin{lem}
\label{lem:close}
For every graph $H$ and every $\delta>0$, there exists a positive integer $h_0$  such that the following holds. If $s(H^{(h)};G^\mu) \ge \frac{1}{2} s(H^{(h)};H)$ for some $h \ge h_0$, then $d_1(G^\mu,H) \le \delta$.
\end{lem}
\begin{proof}
Let $\delta'>0$ be arbitrary. Recall that the factor-graph  $\tilde{H}$ is twin-free, and there exists a positive integer vector $\bk$ such that $H=\tilde{H}^{(\bk)}$.
Note that
$$s(H^{(h)};H) \ge \prod_{v \in V(\tilde{H})} \mu_\bk(v)^{h \bk(v)}.$$
Denote the number of vertices of $\tilde{H}$ by $k$ and let $[H^{(h)},\tilde{H}]$ be the $k$-partially labeled graph based on $H^{(h)}$
where the $k$ labeled vertices induce a copy of $\tilde{H}$. Suppose that $G^\mu$ satisfies $s(H^{(h)};G^\mu) \ge \frac{1}{2}s(H^{(h)};H)$.

By (\ref{eq:averaging}) there exists $\phi:[k] \rightarrow V(G)$ such that
\begin{equation}
\label{eq:assumption}
s([H^{(h)},\tilde{H}], \phi; G^\mu) \ge \frac{1}{2}s(H^{(h)};H)>0.
\end{equation}

Then $G$ must induce a copy of $\tilde{H}$ on $\phi([k])$ as otherwise $s([H^{(h)},\tilde{H}], \phi; G^\mu)=0$.
For every vertex $v \in V(\tilde{H})$, let $A_v \subseteq V(G)$ denote the set of vertices $u$ in $G$ such that $u$ is a twin of $v$ in the subgraph of $G$ induced by $\phi([k]) \cup \{u\}$. Since $\tilde{H}$ is twin-free, the sets $A_v$ are disjoint. Let $a_v :=\mu(A_v)$ for every $v \in V(\tilde{H})$ and $\beta:=1-\sum_{v \in V(\tilde{H})} a_v$. Since every extension of $\phi$ to a strong homomorphism from $H^{(h)}$ to $G$ maps all twins of $v$ in $H^{(h)}$ to vertices in $A_v$, it follows from (\ref{eq:maximizing_product}) that
\begin{eqnarray*}
s([H^{(h)},\tilde{H}], \phi;G^\mu) &\le& \prod_{v \in V(\tilde{H})} a_v^{h \bk(v) - 1} \le (1-\beta)^{h \|\bk\|_1 - |V(\tilde{H})|} \prod_{v \in V(\tilde{H})} \left(\frac{h \bk(v) - 1}{h \|\bk\|_1 - |V(\tilde{H})|}\right)^{h \bk(v) - 1}
\\ &\le& (1-\beta)^{h \|\bk\|_1 - |V(\tilde{H})|} \prod_{v \in V(\tilde{H})} \left(\frac{h \bk(v) - 1}{h \|\bk\|_1 - |V(\tilde{H})|}\right)^{-1}
\prod_{v \in V(\tilde{H})} \mu_\bk^{h \bk(v)}
\\&=& (1-\beta)^{h \|\bk\|_1 - |V(\tilde{H})|} \prod_{v \in V(\tilde{H})} \left(\frac{h \bk(v) - 1}{h \|\bk\|_1 - |V(\tilde{H})|}\right)^{-1} s(H^{(h)};H).
\end{eqnarray*}

Combining this with (\ref{eq:assumption}) we conclude that $\beta \le \delta'$ provided that $h$ is sufficiently large.

So far, we have established that $\beta$, the measure of the vertices that do not belong to any $A_v$ is small provided that $h$ is sufficiently large. Next we show that the measure of present/absent edges between $A_v$'s that are not consistent with a blow-up of $\tilde{H}$ is small.
Let
$$\gamma:=\sum_{v_1,v_2 \in V(\tilde{H})} \sum_{\substack{u_1 \in A_{v_1}, u_2 \in A_{v_2} \\ A_G(u_1,u_2) \neq A_{\tilde{H}}(v_1,v_2)}} \mu(u_1)\mu(u_2).$$
Call an extension of $\phi$ to a map from $V(H^{(h)})$ to $V(G)$ a \emph{good extension} if for every $v \in V(\tilde{H})$, it maps every twin of $v$ to a vertex in $A_v$.  Note that there exist $v_1,v_2 \in V(\tilde{H})$ such that $$\sum_{\substack{u_1 \in A_{v_1}, u_2 \in A_{v_2} \\ A_G(u_1,u_2) \neq A_{\tilde{H}}(v_1,v_2)}} \mu(u_1)\mu(u_2) \ge \frac{\gamma}{k^2}.$$ There are at least $\frac{h}{2}-1$ \emph{disjoint} edges $e_1,\ldots,e_{\frac{h}{2}-1}$ on the unlabeled vertices of $[H^{(h)},\tilde{H}]$ such that one endpoint of every $e_i$ is  a twin of $v_1$ and the other endpoint is a twin of $v_2$. Hence the probability that a random (chosen according to the distribution $\mu$) good extension $\psi:V(H^{(h)}) \rightarrow V(G)$ of $\phi$ is a strong homomorphism is at most
$$\left(1 - \frac{\gamma}{k^2}\right)^{\frac{h}{2}-1}.$$
It follows that
\begin{eqnarray*}
\frac{1}{2}s(H^{(h)};H) &\le&  \left(1 - \frac{\gamma}{ k^2}\right)^{\frac{h}{2}-1} \prod_{v \in V(\tilde{H})} a_v^{h \bk(v) - 1} \\
&\le& \left(1 - \frac{\gamma}{k^2}\right)^{\frac{h}{2}-1}
\prod_{v \in V(\tilde{H})} \left(\frac{h \bk(v) - 1}{h \|\bk\|_1 - |V(\tilde{H})|}\right)^{-1} s(H^{(h)};H),
\end{eqnarray*}
which shows that $\gamma \le \delta'$ and $ \max_{v \in V(\tilde{H})} |a_v-\mu_\bk(v)| \le \delta'$ provided that $h$ is sufficiently large. Taking $\delta'$ to be sufficiently small and bounding $d_1(G^\mu,\tilde{H}^{\mu_\bk}) =d_1(G^\mu,H)$ in terms of $\beta, \gamma$ and $\max_{v \in V(\tilde{H})}|a_v-\mu_\bk(v)|$, we conclude that $d_1(G^\mu,H) \le \delta$ for large enough $h$.
\end{proof}

Lemma~\ref{lem:close} says that if $s(H^{(h)};G^\mu)$ is sufficiently large, then the global structure of $G^\mu$ is similar to $H$. Next we want to use this approximate global structure to deduce some local information about $G^\mu$. The next lemma shows that if the adjacency of a vertex $v_0$ in $G$ does not match this global structure, then $v_0$ is in few induced copies of $H^{(h)}$.

\begin{lem}
\label{lem:vertices_regular}
For every graph $H$, and every $\epsilon>0$, there exist $h_0,\delta>0$ such that the following holds. Suppose  $G^\mu$ and $F^\mu \sim H^\nu$ are commensurable, $d_1(G^\mu,F^\mu) \le \delta$, and $\alpha(\nu)  \ge \epsilon$.  If $v_0 \in V(G)$ is not $\epsilon$-regular with respect to $(G,F,\mu)$, then for every $h \geq h_0$,
$$s(\partial H^{(h)}, 1 \mapsto v_0; G^\mu) \le \epsilon s(\partial H^{(h)}, 1 \mapsto v_0; F^{\mu}).$$
\end{lem}
\begin{proof}
Define $\phi:\{1\} \rightarrow V(G)$  as $\phi(1)=v_0$.
Set $n:=|V(H^{(h)})|$ and $k:=|V(\tilde{H})|$, and let $\bk$ be such that $H=\tilde{H}^{(\bk)}$. We can assume that $\alpha(\mu_\bk) \geq \epsilon$ since the statement of the  lemma becomes stronger as $\epsilon$ decreases. Let $\pi: V(F) \to V(\tilde H)$ be the canonical strong homomorphism  from $F$ to $\tilde{H}$. Fix a vertex $u_0 \in V(H^{(h)})$, and let $H^{(h)}_{u_0}$ be the corresponding $1$-partially labeled graph with $u_0$ labeled.
We consider a random map $\eta:V(H^{(h)}) \to V(G)$ with $\eta(u_0)=v_0$ chosen according to the distribution $\mu$,  and bound the probability of such a map being a strong homomorphism, depending on which of the alternatives listed in Lemma~\ref{lem:technical2} holds when it is applied to $\psi:=\eta \circ \pi|_{V(H^{(h)})}$ and an appropriate choice of $\gamma$.

The parameter  $\gamma$ used in the subsequent calculations is chosen implicitly so that relevant inequalities hold. It suffices to observe that for any $n$ sufficiently large as a function of $\epsilon$, those inequalities hold for $\gamma$ sufficiently small as a function of $\epsilon$ and sufficiently large as a function of $\delta$.

Suppose first that Lemma~\ref{lem:technical2}~{\bf (a)} holds. That is there exists a set $X$ with $|X| = \lfloor \gamma n \rfloor$ so that $\eta \circ \pi |_{V(H^{(h)}) \setminus X}$ is a strong homomorphism from $H^{(h)}-X$ to both $G$ and $F$.

There are $\binom{n}{\lfloor \gamma n \rfloor}$ choices for  $X$. Furthermore every vertex of $H^{(h)}$ has at least $\alpha(\mu_\bk) n$ twins, and thus at least $\alpha(\mu_\bk) n - \gamma n$ of them are in $V(H^{(h)}) \setminus X$.  Then, as $v_0$ is not $\epsilon$-regular,
conditioned on Lemma~\ref{lem:technical2}~{\bf (a)}, the probability that $\eta$ is a strong homomorphism from $H^{(h)}$ to $G$ is at most
$$ (1-\epsilon)^{\alpha(\mu_\bk) n - \gamma n}\binom{n}{\lfloor \gamma n \rfloor} \max_{\substack{X \subseteq V(H^{(h)}) \\ |X| = \lfloor \gamma n \rfloor}} s(H^{(h)}_{u_0}-X, \phi; F^{\mu}).$$
The maximum in the expression above can be in turn upper bounded by $\left(\frac{1}{\alpha(\nu)}\right)^{\gamma n}s(H^{(h)}_{u_0}, \phi; F^{\mu})$, as the probability that any given vertex of $X$ is mapped to extend the existing strong homomorphism from $V(H^{(h)}) \setminus X$ into $F$ is at least $\alpha(\nu)$. It follows that the probability that $\eta$ is a strong homomorphism in this case is at most

\begin{align}
\nonumber (1-\epsilon)^{\alpha(\mu_\bk) n - \gamma n}&\binom{n}{\lfloor \gamma n \rfloor}\left(\frac{1}{\alpha(\nu)}\right)^{\gamma n} s(H^{(h)}_{u_0}, \phi; F^{\mu}) \\ &\leq \exp(\ln(1-\epsilon)\alpha(\mu_\bk) n -\ln(\alpha(\nu))\gamma n - \gamma n -\ln(\gamma) \gamma n) s(H^{(h)}_{u_0}, \phi; F^{\mu}) \notag \\  &< \frac{\epsilon}{2} s(H^{(h)}_{u_0}, \phi; F^{\mu}). \label{eq:closetohom}
\end{align}

If Lemma~\ref{lem:technical2}~{\bf (b)} holds, then there exists sets $Y_1,Y_2 \subseteq V(H^{(h)})$ with $|Y_1|,|Y_2| \geq \frac{\gamma\alpha(\mu_\bk)}{k}n$ so that  for all $y_1 \in Y_1$, $y_2 \in Y_2$ we have $A_{\tilde H}(\psi(y_1),\psi(y_2)) \neq A_{H^{(h)}}(y_1,y_2)$. Let $J^\mu$ be the weighted graph with $V(J):=V(G)$ and $E(J):=E(G) \triangle E(F)$.  Now we can apply the estimate (\ref{eq:bound1}) from Lemma~\ref{lem:technical1} with  $\ell=\frac{\gamma \alpha(\mu_\bk) n}{2k}$. The probability of $\eta$ being a strong homomorphism in this case is at most
\begin{eqnarray}
\nonumber
3^{n}\delta^{\frac{\gamma\alpha(\mu_\bk) n}{2k}} &=& \exp\left( n \ln 3 + \frac{n \gamma \alpha(\mu_\bk)\ln \delta }{2k} \right) \leq \frac{\epsilon^{n+1}}{2} \le \frac{\epsilon}{2}\alpha(\nu)^n \\& \le & \frac{\epsilon}{2} s(H^{(h)}_{u_0}, \phi; H^{\nu})=\frac{\epsilon}{2} s(H^{(h)}_{u_0}, \phi; F^{\mu}).
\label{eq:farfromhom}
\end{eqnarray}

Summing up both (\ref{eq:closetohom}) and (\ref{eq:farfromhom}) over different choices of $u_0 \in V(H^{(h)})$, we conclude that  the lemma holds.
\end{proof}

Finally, the last lemma required in the proof of Theorem~\ref{thm:main} says that if commensurable $G^\mu$ and $F^\mu \sim H^\nu$ are similar in the local sense, then $s(H^{(h)};G^\mu) \leq s(H^{(h)};H^\nu)$.

\begin{lem}\label{lem:exact}
For every graph $H$ and every $\lambda>0$, there exist $\epsilon,h_0>0$ such that the following
holds. Suppose  $G^\mu$ and $F^\mu \sim H^\nu$ are commensurable and $\alpha(\nu)  \ge \lambda$. If for every $v \in V(G)$,
\begin{equation}
\label{eq:epsilonreg}
\mu\left(\{u: A_G(v,u) \neq A_F(v,u)\}\right) \le \epsilon,
\end{equation}
then $s(H^{(h)};G^\mu) \leq s(H^{(h)};H^\nu)$ for every $h \geq h_0$.
\end{lem}

\begin{proof}
Suppose that $H=\tilde{H}^{(\bk)}$, and let $\pi_F:V(F) \to V(\tilde{H})$ and $\pi_{H^{(h)}}:V(H^{(h)}) \to V(\tilde{H})$ be the canonical  strong homomorphisms from $F$ and $H^{(h)}$ to $\tilde{H}$, respectively. Let $J^\mu$ be the weighted graph commensurable with $G^\mu$ and $F^\mu$ with $E(J)=E(G) \triangle E(F)$. By the assumption (\ref{eq:epsilonreg}), we have $s(e,\phi;J^\mu) \leq \epsilon$ for every $\phi: \{1\} \to V(J)$ where $e$ denotes an edge with one of its ends labeled. Denote $\delta := s(K_2;J^\mu)$, and note that there exist $v_1,v_2 \in V(\tilde{H})$ such that
\begin{equation}
\label{eq:deltaBetweenBags}
\Pr[A_J(\mathbf{u}_1,\mathbf{u}_2)=1 \ | \ \mathbf{u}_1 \in \pi^{-1}_F(v_1) \wedge \mathbf{u}_2 \in \pi^{-1}_F(v_2)] \ge \delta,
\end{equation}
where $\mathbf{u}_1$ and $\mathbf{u}_2$ are independent random variables taking values in $V(F)$ according to  $\mu$.

Note that for every strong homomorphism $\theta$ from $H^{(h)}$ to $F$, there exists $\sigma \in \Aut(\tilde{H})$ such that $\pi_{F} \circ \theta = \pi_{H^{(h)}}$. For every $\sigma \in \Aut(\tilde{H})$, let $\mathbf{\theta}_\sigma:V(H^{(h)}) \rightarrow V(F)$ be a random map chosen according to the measure $\mu$ and conditioned on $\pi_{F} \circ \mathbf{\theta}_\sigma = \pi_{H^{(h)}}$. The map $\mathbf{\theta}_\sigma$ can be sampled by mapping every vertex $v \in V(H^{(h)})$ independently and randomly according to $\mu$ to a vertex in $\pi_F^{-1} \circ \pi_{H^{(h)}}(v)$. Note that $\mathbb{\theta}_\sigma$ is a strong homomorphism from $H^{(h)}$ to $F$. If $h>1$, there exists distinct vertices $w_1,w_2 \in V(H^{(h)})$ with $\pi_{H^{(h)}}(w_1)=v_1$ and $\pi_{H^{(h)}}(w_2)=v_2$. Then $\theta_\sigma$ maps $w_1$ and $w_2$ randomly and independently to vertices in $\pi_F^{-1}(v_1)$ and $\pi_F^{-1}(v_2)$, respectively. If $\theta_\sigma$ is a strong homomorphism from $H^{(h)}$ to $G$, then $A_J(\theta_\sigma(w_1),\theta_\sigma(w_2))=0$. Hence by (\ref{eq:deltaBetweenBags}), for $h>1$, the probability that $\mathbb{\theta}_\sigma$ is a strong homomorphism  from $H^{(h)}$ to $G$ is at most $1-\delta$.

Let $p_0$ be the probability that a random map $\psi:V(H^{(h)}) \to
V(G)$  chosen according to the distribution $\mu$ is a strong homomorphism  into both $G$ and $F$. From the above discussion, we conclude that
\begin{equation}\label{eq:p0}
p_0 \leq (1-\delta)s(H^{(h)};F^\mu)=(1-\delta)s(H^{(h)};H^\nu).
\end{equation}
Let $p_1=s(H^{(h)};G^\mu)-p_0$ be the probability that $\psi$ is a strong homomorphism from $H^{(h)}$ into $G$, but not into $F$.  We apply Lemma~\ref{lem:technical2}  to $\pi_F \circ \psi:V(H^{(h)}) \to V(\tilde{H})$ with $\gamma=\frac{1}{2|V(H)|}$. If {\bf (a)} holds, then there exists a set $X \subseteq V(H^{(h)})$ of size at most $h/2$ such that $\pi_F \circ \psi|_{V(H^{(h)}) \setminus X}$ is a strong homomorphism. In particular for every $v \in V(\tilde{H})$, we have
$$\left|\left\{u \in V(H^{(h)}): \mbox{$u$ is a twin of $v$ in $H^{(h)}$,  and $\pi\psi(u)=v$}\right\}\right| \ge h/2.$$
Then since by our assumption $\psi$ is \emph{not} a strong homomorphism into $F$, there exists a vertex $v_0 \in V(H^{(h)})$ such that $$|\{u \in V(H^{(h)}): A_{H^{(h)}}(v_0,u) \neq  A_{F}(\psi(v_0),\psi(u))\}| \ge h/2,$$
where we used the fact that $A_{\tilde{H}}(\pi_F \circ \psi(v_0),\pi_F \circ \psi(u))=A_{F}(\psi(v_0),\psi(u))$.

If Lemma~\ref{lem:technical2}~{\bf (b)} holds, then there are $Y_1,Y_2 \subseteq V(H^{(h)})$ of size at least $\frac{\gamma\alpha(\mu_\bk) h|V(H)|}{|V(\tilde{H})|}=\frac{\alpha(\mu_\bk) h}{2|V(\tilde{H})|}$ so that for all $y_1 \in Y_1$, $y_2 \in Y_2$ we have $ A_{H^{(h)}}(y_1,y_2) \neq A_{\tilde H}(\pi_F \circ \psi(y_1),\pi_F \circ \psi(y_2))=A_F(\psi(y_1),\psi(y_2))$.  In particular there exists a vertex $v_0 \in V(H^{(h)})$ such that
\begin{equation}
\label{eq:largeStar}
\left|\left\{u \in V(H^{(h)}): A_{H^{(h)}}(v_0,u) \neq  A_{F}(\psi(v_0),\psi(u))\right\}\right| \ge \frac{\alpha(\mu_\bk) h}{2|V(\tilde{H})|}.
\end{equation}

So Lemma~\ref{lem:technical2}~{\bf (a)} and~{\bf (b)} both imply the existence of $v_0 \in V(H^{(h)})$ satisfying (\ref{eq:largeStar}).
Now since $\psi$ is a strong homomorphism from $H^{(h)}$ to $G$, it follows that $\psi(V(H^{(h)}))$ contains a subgraph isomorphic to
$K_{1,\ell}$ in $J$ where $\ell=\left\lfloor \frac{\alpha(\mu_\bk) h}{2|V(\tilde{H})|}\right\rfloor$.

By Lemma~\ref{lem:technical1} we now have
\begin{equation}\label{eq:p1}
p_1 \leq  3^{h |V(H)|} \delta \epsilon^{\left\lfloor \frac{\alpha(\mu_\bk) h}{2|V(\tilde{H})|}\right\rfloor}.
\end{equation}
Now (\ref{eq:p0}) and (\ref{eq:p1}) imply that
\begin{eqnarray*}
s(H^{(h)};G^\mu) &\le& (1-\delta)s(H^{(h)};H^\nu) +  3^{h |V(H)|} \delta \epsilon^{\left\lfloor \frac{\alpha(\mu_\bk) h}{2|V(\tilde{H})|}\right\rfloor}
\\ &\le& s(H^{(h)};H^\nu) -\delta \left(\alpha(\nu)^{h|V(H)|} - 3^{h |V(H)|} \epsilon^{\left\lfloor \frac{\alpha(\mu_\bk) h}{2|V(\tilde{H})|}\right\rfloor}\right)
\\&\le& s(H^{(h)};H^\nu) -\delta \left(\lambda^{h|V(H)|} - 3^{h |V(H)|} \epsilon^{\left\lfloor \frac{\alpha(\mu_\bk) h}{2|V(\tilde{H})|}\right\rfloor}\right)
\\&\le& s(H^{(h)};H^\nu)
\end{eqnarray*}
as long as $\epsilon$ is sufficiently small and $h$ is sufficiently large.
\end{proof}%

With Lemmas~\ref{lem:close},~\ref{lem:vertices_regular}, and~\ref{lem:exact} in hand, we can now complete the proof of Theorem~\ref{thm:main}.

\begin{proof}[Proof of Theorem~\ref{thm:main}]
Let $H=\tilde H^{(\bk)}$. Set $\lambda:=\lambda_H$ where $\lambda_H$ is a constant that depends only on $H$ and is specified below. Let $\epsilon < \lambda$ be chosen so that Lemma~\ref{lem:exact} holds for this value of $\lambda$ and some $h_0$. Let $\delta$ be chosen sufficiently small so that
Lemma~\ref{lem:vertices_regular} holds for this value of $\epsilon$. Finally, choose $h_0$ so that
Lemmas~\ref{lem:close},~\ref{lem:vertices_regular}~and~\ref{lem:exact} all hold for the chosen values of $\epsilon$ and $\delta$.

Let $h \ge h_0$ and suppose that there exists a weighted graph $G^\mu$ such that $s(H^{(h)};G^\mu)> \max_{\mu'}s(H^{(h)};H^{\mu'})$.
We can assume  that
\begin{equation}
\label{eq:Gmaximize}
\sup_{F^{\mu'} : \tilde{F} = \tilde{G}}s(H^{(h)};F^{\mu'}) = s(H^{(h)};G^{\mu}).
\end{equation}

Let $\nu$ be such that $\tilde{H}^{\nu}$ minimizes $d_1(G^\mu,\tilde{H}^\nu)$. By~(\ref{eq:Gmaximize}), without loss of generality, we can assume that
$G^\mu$ is commensurable with $F^\mu \sim \tilde{H}^\nu$ satisfying
$$d_1(G^\mu,\tilde{H}^\nu) = \sum_{u,v \in V(G)}\mu(u)\mu(v) |A_{G}(u,v)-A_{F}(u,v)|. $$

By Lemma~\ref{lem:close} we have  $\delta \ge d_1(G^\mu,H) \ge d_1(G^\mu,\tilde{H}^\nu)$ which shows that $d_1(H,\tilde{H}^\nu) \le 2\delta$. Hence comparing $s(\tilde{H}; H)$ and $s(\tilde{H}, \tilde{H}^\nu)$ we conclude from (\ref{eq:continuity}) that there exists a constant $\lambda_H>0$ depending only on $H$ such that $\alpha(\nu) > \lambda_{H}$ provided that $\delta$ is sufficiently small.

It follows from  (\ref{eq:Gmaximize}) and (\ref{eq:partial}) that $s(\partial H^{(h)}, \phi; G^\mu)$ is independent of the choice of
$\phi:\{1\} \to V(G)$. Then Lemma~\ref{lem:vertices_regular} shows that  every vertex $v \in V(G)$ is
$\epsilon$-regular  with respect to $(G,F,\nu)$ as otherwise one would have
\begin{eqnarray*}
s(H^{(h)};G^\mu) &=& \frac{1}{h|V(H)| }\Ex_{\phi} \left[ s(\partial H^{(h)},\phi;G^\mu) \right] \le \frac{ \epsilon}{h|V(H)| } \Ex_{\phi} \left[ s(\partial H^{(h)},\phi;F^\mu) \right] \\ &=& \epsilon s(H^{(h)};F^\mu) = \epsilon s(H^{(h)};\tilde{H}^\nu),
\end{eqnarray*}
where both expectations are over a random $\phi:\{1\} \to V(G)$ chosen according to $\mu$.
In order to  apply Lemma~\ref{lem:exact}, we need to show there is no vertex $v \in V(G)$ with
\begin{equation}
\label{eq:irregular}
\mu \left(\{u: A_G(u,v) \neq A_F(u,v)\} \right) > \epsilon.
\end{equation}
Suppose to the contrary that (\ref{eq:irregular}) holds for a vertex $v \in V(G)$. Since $v$ is $\epsilon$-regular, there exists
$v_0 \in V(G)$ such that $$\mu\left(\{ u \in V(G) : A_G(u,v) \neq A_F(u,v_0)\}\right) \leq \epsilon.$$
Let $F_1$ be the graph obtained from $F$ by replacing $v$ with a copy of $v_0$. By (\ref{eq:irregular}) we have $d_1(G^\mu,F_1^\mu) < d_1(G^\mu,F^\mu)$. This contradicts the assumption that $\tilde{H}^{\nu}$ minimizes $d_1(G^\mu,\tilde{H}^\nu)$ since $F_1^\mu \sim \tilde{H}^{\nu'}$ for some probability measure $\nu'$. Hence we can apply Lemma~\ref{lem:exact}  to conclude that  $s(H^{(h)};G^\mu) \le s(H^{(h)};F^\mu)= s(H^{(h)};\tilde{H}^\nu)$.
\end{proof}

\section{Proof of Theorem~\ref{thm:balanced}}
Suppose that $H=\tilde{H}^{(\bk)}$. Since $s(H^{(h)};H)=s(H^{(h)};\tilde{H}^{\mu_\bk})=s(\tilde{H}^{(h \bk)};\tilde{H}^{\mu_\bk})$, Theorem~\ref{thm:balanced} follows from Theorem~\ref{thm:main} and Lemma~\ref{lem:balanced} below.

\begin{lem}
\label{lem:balanced}
Let $H$ be a twin-free graph and $\bk \in \mathbb{N}^{V(G)}$ be a vector. Then there exists $h_0>0$ such that $s(H^{(h \bk)};H^{\mu})$ achieves a  maximum
on $\mathcal{M}(V(H))$ at $\mu:=\mu_{\bk}$ for all integers $h \geq h_0$ if and only if $\bk$ is $\Aut(H)$-invariant.
\end{lem}

\begin{proof}
Note that for every probability distribution $\mu$ on $V(H)$, we have
$$s(H^{\bk};H^{\mu}) =  \sum_{\sigma \in \Aut(H)} p_{\sigma(\bk)}^h(\mu),$$
where $p_{\sigma(\bk)}(\mu)$ is defined as in (\ref{eq:s_ba}).

If $\bk$ is $\Aut(H)$-invariant, then by (\ref{eq:maximizing_product}), for every $\mu \in \mathcal{M}(V(H))$, we have $$s(H^{\bk};H^{\mu})=|\Aut(H)|p_{\bk}^h(\mu) \leq |\Aut(H)|p_{\bk}^h(\mu_\bk)=s(H^{\bk};H^{\mu_\bk}).$$
Hence in such a case $s(H^{\bk};H^{\mu})$ achieves its maximum
on $\mathcal{M}(V(H))$ at $\mu_{\bk}$ for all positive $h$.

Suppose now that $\mu:=\mu_{\bk}$ is a local maximum of $s(H^{\bk};H^{\mu})$ on $\mathcal{M}(V(H))$  for all integers $h \geq h_0$. For every non-negative real $r$, let $$A_r := \{\sigma \in \Aut(H) \: | \: p_{\sigma(\bk)}(\mu_{\bk}) = r\}.$$
We have
$$s(H^{\bk};H^{\mu})=\sum_{r \in \mathbb{R}_+} \sum_{\sigma \in A_r} p_{\sigma(\bk)}^h(\mu).$$
For every $\sigma \in A_r$, we have
$$\left. \frac{\partial p_{\sigma(\bk)}^h}{\partial \mu(v)}\right|_{\mu_\bk}:=h\bk(\sigma(v)) \frac{\prod_{u \in V(H)}\mu_\bk(u)^{h\bk(\sigma(u))}}{\mu_\bk(v)} = h\bk(\sigma(v)) \frac{r^h}{\mu_\bk(v)}=hr^h \|\bk\|_1 \frac{\bk(\sigma(v))}{\bk(v)}.$$
Since $\mu_{\bk}$ is a local maximum of $s(H^{\bk};H^{\mu})$, we know that
$$\left. \frac{\partial s(H^{\bk};H^{\mu})}{\partial \mu(v)}\right|_{\mu_\bk} = \sum_{r \in \mathbb{R}_+} \sum_{\sigma \in A_r} \left. \frac{\partial p_{\sigma(\bk)}^h}{\partial \mu(v)}\right|_{\mu_\bk}
=\sum_{r \in \mathbb{R}_+}hr^h \|\bk\|_1 \left(\frac{\sum_{\sigma \in A_r} \bk(\sigma(v))}{\bk(v)} \right) $$
is independent of the choice of $v \in V(H)$ for all integers $h \geq h_0$. It follows that for every $r>0$, $a_{r,v}:=\frac{\sum_{\sigma \in A_r} \bk(\sigma(v))}{\bk(v)}$ is  independent of the choice of $v \in V(H)$. If $\bk$ is not $\Aut(H)$-invariant, then we can choose $r \in \mathbb{R}_+$ and $u \in V(H)$ so that $\bk(u) > \bk(\sigma(u))$ for some $\sigma \in A_r$, and subject to this condition $\bk(u)$ is maximum. Then $\bk(u) \geq \bk(\sigma(u))$ for every $\sigma \in \Aut(H)$. Therefore, $a_{r,u} < |A_r|$. On the other hand, if $u'$ is chosen with $\bk(u')$ minimum, then $a_{r,u'} \geq |A_r|$, a contradiction. It follows that $\bk$ is  $\Aut(H)$-invariant, as desired.
\end{proof}

\section{Concluding remarks}

In this paper we have shown that the strong homomorphism density of
sufficiently large balanced blow-ups of a fixed graph $H$ is maximized by $H^{\nu}$ for an appropriate probability measure $\nu$. It
seems likely that similar techniques might allow one to refine and extend this result to other settings, described below.  Our choice to focus on weighted graphs and on strong homomorphism densities, rather than induced subgraph counts, was motivated by the fact that presenting the
argument in this setting avoids multiple unnecessary technicalities and more importantly, by our belief that this setting is natural for the
problem. It is for similar reasons that in recent years the rapidly developing theory of graph homomorphisms has been replacing the classical
approach of considering subgraph counts.

\subsection{Counting induced copies} As mentioned in the introduction, Bollob\'as \emph{et al.}~\cite{MR1326923} asked the following question.

\begin{question}\label{q:Bollobas}
For which graphs $H$, does there exist a constant $h$ such that for
sufficiently large $n$, the graph on $n$ vertices that contains the
maximum number of induced copies of $H^{(h)}$ is a blow-up of $H$?
\end{question}

We convinced ourselves that such a constant exists for all graphs $H$.
Theorem~\ref{thm:main} gives an asymptotic answer to this question.
There exist techniques which in some other extremal problems allow one
to deduce exact results from the asymptotic ones, employed for example
in~\cite{CycleErdos}. We were unable to do so for Question~\ref{q:Bollobas}. One
can however adjust the statements and proofs of
Lemmas~\ref{lem:vertices_regular} and~\ref{lem:exact} to provide the
answer to Question~\ref{q:Bollobas}. The arguments involving
derivatives in the proof of Theorem~\ref{thm:main} can be modified to
use induction instead. Additional adjustments are necessary to control
the error terms which appear when one considers embeddings instead of
strong homomorphisms.

\subsection{A hypergraph generalization.}  
A \emph{finite palette} is a sequence $K = (K_j)_{j=0}^{\infty}$ of finite sets, so that $0 \in K_j$ for every $j$ and
 $K_j=\{0\}$ for all but finitely many values of $j$.
Let $V$ be a finite set. For a positive integer $j$, let $V^{[j]}$ denote
the set of all maps $\phi:[j] \to V$ and let $\Inj([j],V) \subseteq V^{[j]}$ denote the set of all the injective maps.  Slightly modifying the definition from~\cite{MR2666763}, we define a \emph{$K$-colored  hypergraph} on the vertex set $V$ to be a
tuple $G = (G_j)_{j=0}^{\infty}$, where each $G_j :  V^{[j]} \to K_j$ is a function.

Undirected graphs fit this framework as follows.  For any $k \geq 0$, let the \emph{monochromatic palette $\{0, 1\}_k$ of order $k$}  be the palette whose $k$-th component is $\{0, 1\}$ and all other components are trivial. We say that the hypergraph $G$ is \emph{undirected} if $G_j(\phi \circ \sigma)=G_j(\phi)$ for every $j \geq 0$, every $\sigma \in \Inj([j],[j])$ and every $\phi \in V^{[j]}$. Finally, we say that $G$ is \emph{loopless} if $G_j(\phi) \neq 0$ implies $\phi \in \Inj([j],V)$ for every  $j \geq 0$ and every $\phi \in V^{[j]}$. Then every graph can be viewed as a $\{0,1\}_2$-colored undirected loopless hypergraph.

Let $K$ be a finite palette, and let $G$ and $H$ be two $K$-colored hypergraphs with vertex sets $V_G$ and $V_H$ respectively.  We say that a map $\psi: V_G \to V_H$ is a \emph{strong homomorphism} if $H_j(\psi \circ \phi)=G_j(\phi)$ for every $j \geq 0$ and $\phi:[j] \to V_G$. Given a positive integer $h$, we say that $G$ is an \emph{$h$-blowup} of $H$ if there exists a strong homomorphism $\psi: V_G \to V_H$ such that $|\psi^{-1}(v)|=h$ for every $v \in V(H)$. It is not hard to see that if $H$ is loopless then its $h$-blowup is unique up to an automorphism. We denote it by $H^{(h)}$. Definitions of \emph{weighted} colored hypergraphs and of the \emph{strong homomosphism density} $s(G;H^{\nu})$, where $\nu$ is a probability measure, are completely analogous to the definitions given in Section~\ref{sec:weighted}.

We believe that the following direct analogue of Theorem~\ref{thm:main} holds in this setting and can be obtained by modifying the proof presented in this paper. Let $K$ be a finite palette and let $H$ be a loopless $K$-colored hypergraph with the vertex set $V$. Then there exists a positive integer $h_0$ such that for every $h \ge h_0$,  there exists a distribution $\nu \in \mathcal{M}(V)$ with
$$s(H^{(h)};H^{\nu}) = \sup_{G} s(H^{(h)};G),$$
where the supremum is over all  loopless $K$-colored hypergraphs $G$.

Perhaps, an even more general version of Theorem~\ref{thm:main} can
be stated in category theoretic language.

\section*{Acknowledgements}
The first author wishes to thank Igor Shinkar for drawing his attention to the reference~\cite{MR1326923}.

\bibliographystyle{alpha}
\bibliography{blowup}

\end{document}